\documentclass{amsart}
\usepackage{amsfonts,amssymb,latexsym,amsthm,amsmath}
\usepackage{mathrsfs}
\usepackage[PostScript=dvips]{diagrams}
\usepackage{euscript}

\newtheorem{theorem}{Theorem}[section]

\newtheorem{lemma}[theorem]{Lemma}

\newtheorem{prop} [theorem]{Proposition}

\newtheorem{cor}[theorem]{Corollary}

\def\Z{\mathbb{Z}}
\DeclareRobustCommand{\stirling}{\genfrac\{\}{0pt}{}}

\DeclareMathAlphabet\mathbb{U}{msb}{m}{n}

\usepackage{tikz}

\usetikzlibrary{matrix}
\usepackage[arrow,matrix,curve]{xy}\SilentMatrices
\def\xyma{\xymatrix@M.7em}

\def\square{\vbox{
      \hrule height 0.4pt
      \hbox{\vrule width 0.4pt height 5.5pt \kern 5.5pt \vrule width 0.4pt}
      \hrule height 0.4pt}}

\def\ch\mathrm{c h}
\def\ab{\mathrm{a b}}

\def\im{\mathrm{I m}}

\def\coalg{\mathrm{co a l g}}

\def\ad{\mathrm{ad}}
\def\ScriptB{\mathrm{\EuScript{B}}}

\long\def\symbolfootnote[#1]#2{\begingroup%
\def\thefootnote{\fnsymbol{footnote}}\footnote[#1]{#2}\endgroup}

\newcommand{\calH}{\ensuremath{\mathcal{H}}}

\let\la=\langle
\let\ra=\rangle

\numberwithin{equation}{section}

\begin{document}
\title{A combinatorial approach to the exponents of Moore spaces}

\author{Frederick R. Cohen}
\address{Department of Mathematics, University of Rochester, Rochester, NY 14625, USA}
\email{cohf@math.rochester.edu}

\author{Roman Mikhailov}
\address{Chebyshev Laboratory, St. Petersburg State University, 14th Line, 29b,
Saint Petersburg, 199178 Russia and St. Petersburg Department of
Steklov Mathematical Institute} \email{rmikhailov@mail.ru}
\author{Jie Wu }
\address{Department of Mathematics, National University of Singapore, 10 Lower Kent Ridge Road, Singapore 119076} \email{matwuj@nus.edu.sg}
\urladdr{www.math.nus.edu.sg/\~{}matwujie}

\thanks{The main result (Theorem~\ref{theorem1.1})  is supported by Russian Scientific Foundation, grant N 14-21-00035. The last author is also partially supported by the Singapore Ministry
of Education research grant (AcRF Tier 1 WBS No. R-146-000-190-112)
and a grant (No. 11329101) of NSFC of China.}

\begin{abstract}
In this article, we give a combinatorial approach to the exponents of the Moore spaces. Our result states that the projection of the $p^{r+1}$-th power map of the loop space of the $(2n+1)$-dimensional mod $p^r$ Moore space to its atomic piece containing the bottom cell $T^{2n+1}\{p^r\}$ is null homotopic for $n>1$, $p>3$ and $r>1$. This result strengthens the classical result that $\Omega T^{2n+1}\{p^r\}$ has an exponent $p^{r+1}$.
\end{abstract}

\maketitle

\section{Introduction}

The purpose of this article is to give a combinatorial approach to the exponents of Moore spaces.  The exponent problem has been studied by various people with fruitful results~\cite{Barratt, CMN,  CMN3, James,  Neisendorfer, Neisendorfer3, Selick, Theriault, Toda} by using traditional methods. Our approach to the exponent of Moore spaces will be given by studying the combinatorics of the Cohen groups introduced in~\cite{Cohen2} together with collecting the minimal geometric information such as the classical Cohen-Moore-Neisendorfer decompositions and basic properties on the mod $p^r$ homotopy groups of mod $p^r$ Moore spaces~\cite{CMN,CMN2,CMN3}.

Let us begin with a brief review on the Cohen groups. Let $X$ be a pointed space. Recall that the James construction $J(X)$ is the free monoid generated by $X$ subject to the single relation the basepoint $\ast\sim 1$, with weak topology. The James filtration $J_n(X)$ is given by the word length filtration of $J(X)$. Thus $J_n(X)$ is a quotient space of the $n$-fold Cartesian product $X^{\times n}$ as the coequalizer of the coordinate inclusions $d^i\colon X^{n-1}\to X^n$, $(y_1,\ldots,y_{n-1})\mapsto (y_1,\ldots, y_{i-1},\ast,y_i,\ldots,x_{y-1})$ for $1\leq i\leq n$. An important property of the James construction is that $J(X)$ is weakly homotopy equivalent to $\Omega\Sigma X$ if $X$ is path-connected~\cite{James1}. By using the James construction, one can get a combinatorial approach to the self-maps of loop suspensions in the following way. Let $F_n=\la x_1,\ldots,x_n\ra$ be the free group of rank $n$ with a fixed choice of basis $x_1,\ldots,x_n$. Observe that the multiplication of $\Omega\Sigma X$ induces a group structure on $[X^{\times n}, \Omega\Sigma X]$. Consider the naive representation
$$
\tilde e_X\colon F_n\longrightarrow [X^{\times n}, \Omega \Sigma X]
$$
as a group homomorphism, which sends $x_i$ to the homotopy class of the composite
$$
X^{\times n}\rTo^{\pi_i} X\rInto^E\Omega\Sigma X,
$$
where $\pi_i$ is the $i$-th coordinate projection and $E$ is the canonical inclusion. It was discovered in~\cite{Cohen2} that for any co-$H$-space $X$,
$$
\tilde e_X([[x_{i_1},x_{i_2}],\ldots,x_{i_t}])=1
$$
if $i_p=i_q$ for some $1\leq p<q\leq t$. The group $K_n=K_n(x_1,\ldots,x_n)$ was then introduced as the quotient group of $F_n$ subject to the above relations, with the property that $\tilde e_X$ induces a representation
$$
e_X\colon K_n\longrightarrow [X^{\times n}, \Omega\Sigma X]
$$
for any co-$H$-space $X$. In order to obtain self-maps of $\Omega\Sigma X$, the suspension splitting theorem of the James construction gives a good property that the quotient map $q_n\colon X^{\times n}\to J_n(X)$ induces a group monomorphism $q_n^*\colon [J_n(X),\Omega\Sigma X]\to [X^{\times n},\Omega\Sigma X]$ and its image is given by the equalizer of the group homomorphisms $d^{i*}\colon [X^{\times n-1},\Omega\Sigma X]\to [X^{\times n}, \Omega\Sigma X]$ for $1\leq i\leq n$. Moreover
$$
[\Omega\Sigma X, \Omega\Sigma X]\cong [J(X),\Omega\Sigma X]=\lim_n[J_n(X),\Omega\Sigma X]
$$
is the inverse limit for any path-connected space $X$. The interpretation of $d^{i*}$ in the Cohen group $K_n$ is the projection homomorphism
$$
d_i\colon K_n\longrightarrow K_{n-1}
$$
with $d_i(x_j)=x_j$ for $j<i$, $d_i(x_i)=1$ and $d_i(x_j)=x_{j-1}$ for $j>i$. Let $H_n$ be the subgroup of $K_n$ given as the equalizer of the group homomorphisms $d_i$ for $1\leq i\leq n$. The restriction of $e_X$ on the subgroup $H_n$ gives a representation
$$
e_X\colon H_n\longrightarrow [J_n(X),\Omega\Sigma X]
$$
for any co-$H$-space $X$. With taking inverse limit, let $H=\lim\limits_n H_n$, one get a representation
$$
e_X\colon H\longrightarrow[J(X),\Omega\Sigma X]\cong [\Omega\Sigma X,\Omega\Sigma X]
$$
for any path-connected co-$H$-space $X$.

We should point out that the group $K_n$ is isomorphic to Milnor's reduced free group introduced in his fundamental work on homotopy link theory~\cite{Milnor}. A recent application of the group $K_n$ in $4$-manifolds is given in~\cite{Freedman}. The importance of the Cohen groups $K_n$, $H_n$ and $H$ in homotopy theory is that $H$ is a subgroup of the group of self natural transformations of the functor $\Omega\Sigma$ on path-connected co-$H$-spaces with its algebraic version through the Hurewicz homomorphism given exactly by the group of self natural transformations of the tensor algebra functor free abelian groups to coalgebras~\cite{Selick-Wu,Selick-Wu2,Wu}. In particular, some fundamental objects in unstable homotopy theory, that is the Hopf invariants, the Whitehead product, the power maps and the loop of degree maps, are under controlled by the group $H$.

Suppose that the inclusion map $E\colon X\to \Omega\Sigma X$ has a finite order of $p^r$ in the group $[X,\Omega\Sigma X]$. Then the representation $e_X\colon K_n\to [X^{\times n},\Omega \Sigma X]$ factors through the group $K_n^{\Z/p^r}=K_n^{\Z/p^r}(x_1,\ldots,x_n)$, which is the quotient group of $K_n$ by requiring $x_i^{p^r}=1$ for $1\leq i\leq n$.  Similar to the integral version, the equalizer of the operations $d_i$ on $K_n^{\Z/p^r}$ gives the subgroup $H_n^{\Z/p^r}$. The Cohen group $K_n^{\Z/p^r}$ services for the exponent problem, which is under exploration in this article. Observe that the particular element $\alpha_n=x_1x_2\cdots x_n\in H_n^{\Z/p^r}\leq K_n^{\Z/p^r}$ has the geometric interpretation as the homotopy class of the inclusion map $J_n(X)\to \Omega\Sigma X$. Suppose that $\alpha_n^{p^t}=1$ in $K_n^{\Z/p^r}$. Then geometrically it means that the inclusion map $J_n(X)\to\Omega\Sigma X$ has an order bounded by $p^t$ in the group $[J_n(X),\Omega\Sigma X]$. In particular the homotopy groups $\pi_*(\Omega\Sigma X)=\pi_{*+1}(\Sigma X)$ has an exponent bounded by $p^t$ up to the range controlled by $J_n(X)$, namely below $(n+1)$ times the connectivity of $X$. When $n=1$, $\alpha_1^{p^r}=1$, which is the starting point. When $n$ increases, the exponent of $\alpha_n$ also increases. For understanding the growth of $\alpha_n$, it important and fundamental to understand the element $\alpha_n^{p^r}$ and the difference between $\alpha_{n+1}^{p^{r+1}}$ and $\alpha_n^{p^{r+1}}$. By using techniques in group theory, Lemma~\ref{lemma2.6} gives a description of the element $\alpha_n^{p^r}$ and Proposition~\ref{proposition2.7} gives a description of the difference between $\alpha_{n+1}^{p^{r+1}}$ and $\alpha_n^{p^{r+1}}$. Here, we should make a comment that the Stirling number appears naturally in this topic by Lemma~\ref{lemma2.2}.

It should be pointed out that, for any connected space $X$ with a nontrivial reduced homology with coefficients in $p$-local integers, any power $p^t\colon \Omega\Sigma X\to \Omega\Sigma X$ is essential by~\cite[Theorem 3.10]{CMN3}. This property seems to discourage the study on the exponents of the single loop spaces. However, with taking the observation that $\Omega\Sigma X$ has various decompositions, one can ask the following question. Let $T$ be the atomic retract of $\Omega\Sigma X$ containing the bottom cell. Is it possible that there is a choice of the projection map $\pi\colon \Omega\Sigma X\to T$ such that the composite
$$
\Omega\Sigma X\rTo^{p^t}\Omega\Sigma X\rTo^{\pi} T
$$
is null homotopic for some $t$?

By using combinatorial approach, we give the affirmed answer to the above question for Moore spaces. Our result is as follows. Recall ~\cite[Corollary 1.9]{CMN3} that there is a homotopy decomposition
$$
\Omega P^{2n+1}(p^r)\simeq T^{2n+1}\{p^r\}\times \Omega P(n,p^r)
$$
for $p>2$ and $n\geq2$, where $P^{m}(p^r)=S^{m-1}\cup_{p^r}e^m$, the $m$-dimensional mod $p^r$ Moore space, $P(n,p^r)$ is a wedge of mod $p^r$ Moore spaces, and $T^{2n+1}\{p^r\}$ is the atomic retract of $\Omega P^{2n+1}(p^r)$.

\begin{theorem}\label{theorem1.1}
There is a choice of the projection $\partial\colon \Omega P^{2n+1}(p^r)\to T^{2n+1}\{p^r\}$ such that composite
$$
\Omega P^{2n+1}(p^r)\rTo^{p^{r+1}}\Omega P^{2n+1}(p^r)\rTo^{\partial} T^{2n+1}\{p^r\}
$$
is null homotopic for $p>3$, $n>1$ and $r>1$.
\end{theorem}

This theorem strengthens the classical result ~\cite{Neisendorfer} that $\Omega T^{2n+1}\{p^r\}$ has exponent $p^{r+1}$ in the sense that $T^{2n+1}\{p^r\}$ already has exponent $p^{r+1}$ in the above sense.

The article is organized as follows. In section~\ref{section2}, we explore the combinatorics of the Cohen groups. We give some remarks for potential applications for general spaces in section~\ref{section3}. In section~\ref{section4}, we give the applications to the Moore spaces. Theorem~\ref{theorem1.1} is Theorem~\ref{theorem4.1}. In Section~\ref{section5}, we give the applications to the Anick spaces.

\section{Combinatorics of the Cohen groups}\label{section2}
In this section, $p$ is an odd prime and $r\geq 1$.  For elements
$x,y, g_1,\dots, g_k$ of a group, we will use the standard
commutator and left-normalized notation:
$$[x,y]:=x^{-1}y^{-1}xy,\ \ x^y:=y^{-1}xy,\ \ [g_1,\dots, g_k]:=[[g_1,\dots, g_{k-1}],g_k].$$
For $i\geq 1$, we will use the following notation for the
left-Engel brackets
$$[x,_1y]:=[x,y],\ [x,_iy]=[[x,_{i-1}y],y].$$

For $n\geq 1$, the Cohen group $K_n^{\mathbb Z/p^r}=K_n^{\mathbb
Z/p^r}(x_1,\dots, x_n)$ is the quotient of a free group
$F(x_1,\dots, x_n)$ of rank $n$ by all left-normalized commutators
$$
[x_{i_1},\dots, x_{i_k}],\ \text{such that}\ i_s=i_t\ \text{for
some}\ 1\leq s,t\leq n,\ s\neq t
$$
together with $p^r$th powers of generators $x_i^{p^r},\ i=1,\dots,
n$. The group $K_n^{\mathbb Z/p^r}$ is nilpotent of class $n$.

In this paper, we will consider also the following subgroup of
$K_n^{\mathbb Z/p^r}$. Let $\EuScript B_n$ be the subgroup of
$K_n^{\mathbb Z/p}$ generated by all brackets
$$
[x_{i_1},\dots, x_{i_k}],\ k\neq p^t,\ t\geq 0.
$$
For any configuration of brackets $[[...],[[...]...]...]$, in a
commutator of length $k$ whose entrances are generators
$\{x_1,\dots, x_n\}$ only, can be written as a product of
left-normalized commutators of length $k$ with generators as
entrances. This follows from the definition of $K_n^{\mathbb
Z/p^r}$ and the Hall-Witt identity. Therefore, any commutator of
length $\neq p^t, t\geq 0$ whose entrances are generators, is in
$\EuScript B_n$. Obviously, $\EuScript B_n$ is not normal in
$K_n^{\mathbb Z/p}$.

The commutator calculus in groups $K_n^{\mathbb Z/p^r}$ are much
simpler than in free nilpotent groups.  We will need the following
standard relations.

\begin{lemma}\label{firstlemma}
Let $x$ be an element from the generating set $\{x_1,\dots, x_n\}$
and $g$ any element of $K_n^{\mathbb Z/p^r}$. Then, for $k\geq 1$,
\begin{align}
& [x,g^k]=\prod_{i=1}^k[x,_ig]^{\binom{k}{i}};\label{q1}\\
&
(gx)^k=g^kx^k\prod_{i=1}^{k-1}[x,_ig]^{\binom{k}{i+1}}.\label{q2}
\end{align}
\end{lemma}
\begin{proof}
First we prove (\ref{q1}). For $k=1$, this is obvious. Suppose
that the formula is proved for a given $k$. Then, using the
property of the group, that for all elements $h_1,h_2$, $[x,h_1]$
and $[x,h_2]$ commute, we get
\begin{multline*}
[x,g^{k+1}]=[x,g][x,g^k]^{g}=[x,g][x,g^k][x,g^k,g]=\\
[x,g]^{k+1}[x,_{k+1}g]\prod_{i=2}^k[x,_ig]^{\binom{k}{i}+\binom{k}{i-1}}=\prod_{i=1}^{k+1}[x,_ig]^{\binom{k+1}{i}}.
\end{multline*}
The needed relation is proved.

To prove (\ref{q2}), we also use the induction on $k$. For $k=1$
it is obvious. Suppose that (\ref{q2}) is proved for a given $k$.
Then, using the relation $[x^k,g]=[x,g]^k$, we obtain
\begin{multline*}
(gx)^{k+1}=(gx)^k(gx)=g^kx^k(\prod_{i=1}^{k-1}[x,_ig])gx=\\
g^{k+1}x^{k+1}[x^k,g](\prod_{i=1}^{k-1}[x,_ig]^{\binom{k}{i+1}})\prod_{i=1}^{k-1}[x,_{i+1}g]^{\binom{k}{i+1}}=\\
g^{k+1}x^{k+1}[x,g]^{\binom{k+1}{2}}(\prod_{i=2}^{k-1}[x,_ig]^{\binom{k}{i+1}+\binom{k}{i}})[x,_kg]=g^{k+1}x^{k+1}\prod_{i=1}^k[x,_ig]^{\binom{k+1}{i+1}}.
\end{multline*}
The inductive step is done.
\end{proof}

For the convenience, we will work now in the group
$K_{n+1}^{\mathbb Z/p^r}=K_{n+1}^{\mathbb Z/p^r}(x_1,\dots,
x_{n+1}).$ Observe that, for $l>n$,
$$
[x_{n+1},_l(x_1\dots x_n)]=1.
$$
This follows from the simple observation that $K_{n+1}^{\mathbb
Z/p^r}$ is nilpotent of class $n+1$. To describe the commutator
$[x_{n+1},_l(x_1\dots x_n)]$ for $n\geq l$, we will need some
special sets of permutations.

For a given $1\leq l\leq n$, consider the set of permutations of
$\{1,\dots,n\}$ \begin{multline*} \Sigma_l^n=\{(i_1,\dots,
i_{k_1},
i_{k_1+1},\dots, i_{k_2},\dots, i_{k_{l-1}+1},\dots, i_{k_l})\ |\\
i_{k_i+1}<\dots<i_{k_{i+1}},\  k_0=0,\ i=1,\dots, l-1\}
\end{multline*}
That is, $\Sigma_l^n$ consists of permutations on $n$ letters such
that they can be divided into $l$ monotonic blocks. Some
permutations can be divided into $l$ monotonic blocks in different
ways, for a permutation $\sigma$, the number of such divisions we
denote by $d_l(\sigma).$ For example, here is the list of
permutations from $\Sigma_2^3$ with values of $d_2$:
\begin{align*}
& \text{permutation}\ & d_2\\
& (1,2,3) & 2\\
& (2,1,3) & 1\\
& (2,3,1) & 1\\
& (3,1,2) & 1\\
& (1,3,2) & 1\\
& (3,2,1) & 0
\end{align*}
The following proposition follows immediately from the definition
of the set $\Sigma_l^n$.
\begin{lemma}\label{lemma2.2}
$\sum_{\sigma\in \Sigma_l^n}d_l(\sigma)=l!\stirling{n}{l}.$ Here
$\stirling{n}{l}$ is the second Stirling number.
\end{lemma}
Indeed, the Stirling number $\stirling{n}{l}$ is the number of
ways to divide the set $\{1,\dots, n\}$ into $l$ non-empty
subsets. In each of $l$ subsets we order the elements in the
monotonic way. In this partition we can permute all $l$ monotonic
blocks. Each permutation $\sigma$ appears in this way exactly
$d_l(\sigma)$ times.

We will use later one more notation. For $1\leq i\leq n$, denote
$$
\Sigma_l^n(i)=\{(i_1,\dots, i_n)\in \Sigma_l^n |\ i_1=i\}.
$$
\begin{lemma}\label{lemma2}
For any $i$, $\sum_{\sigma\in \Sigma_l^n(i)}d_l(\sigma)$ divides
$(l-1)!$.
\end{lemma}

Lemma \ref{lemma2} follows immediately from the definition of the
set $\Sigma_l^n(i)$. If we consider some permutation from
$\Sigma_l^n(i)$, we can fix the first monotonic block which starts
with $i$ and permute other $(l-1)$ monotonic blocks. One can
easily prove explicit values of the above sum for some $i$-s. For
example, $$\sum_{\sigma\in
\Sigma_l^n(1)}d_l(\sigma)=\stirling{n}{l}(l-1)!,\ \sum_{\sigma\in
\Sigma_l^n(n)}d_l(\sigma)=\stirling{n-1}{l-1}(l-1)!$$ We will
naturally extend the notation $\Sigma_l^n$ for permutations on $n$
(ordered) symbols, for example, for $N>n$, $\sigma\subset
\{1,\dots, N\},$ we say that $\sigma\in \Sigma_l^n$ if it can be
divided into $l$ monotonic blocks. In a natural way, for these
extended cases, one can define $d_l(\sigma)$.

Now we are able to describe the commutators $[x_{n+1},_lx_1\dots
x_n]$.
\begin{lemma}
For any $l\geq 1$ and $n\geq l$,
\begin{equation}\label{decomposition}
[x_{n+1},_lx_1\dots x_n]=\prod_{i=l}^n\prod_{\sigma \in
\Sigma_l^i, \sigma \subseteq \{1,\dots, n\}}[x_{n+1},
x_{\sigma(1)},\dots, x_{\sigma(i)}]^{d_l(\sigma)}.
\end{equation}
\end{lemma}
\begin{proof}
The proof is straightforward, by induction on $l$. For $l=1$, we
have
$$
[x_{n+1}, x_1\dots
x_n]=\prod_{i=1}^n\prod_{j_1<\dots<j_i}[x_{n+1},x_{j_1},\dots,x_{j_i}].
$$
The sets $\Sigma_1^i$ have a single permutation $(1,\dots, i)$. In
the notation used in the formulation of lemma, the product over
such sets means exactly the product over ordered sets of $i$
elements from $\{1,\dots, n\}.$ That is, we have the needed
formula for $l=1$. Now assume that it is proved for a given $l$.
We have \begin{multline*} [x_{n+1},_{l+1}(x_1\dots
x_n)]=[[x_{n+1},_l(x_1\dots x_n)],x_1\dots x_n]=\\
\prod_{i=l}^n\prod_{\sigma \in \Sigma_l^i, \sigma \subseteq
\{1,\dots, n\}}[[x_{n+1}, x_{\sigma(1)},\dots,
x_{\sigma(i)}],x_1\dots x_n]^{d_l(\sigma)}.
\end{multline*}
For a fixed $\sigma\in \Sigma_l^i$ on letters $j_1,\dots, j_i$,
consider the commutator
$$
[x_{n+1}, x_{\sigma(1)},\dots, x_{\sigma(i)},x_1\dots x_n].
$$
Opening this commutator, we get
\begin{equation}\label{bigprod}
[x_{n+1}, x_{\sigma(1)},\dots, x_{\sigma(i)}, x_1\dots
x_n]=\prod_{q_1<\dots<q_t}[x_{n+1},x_{\sigma(1)},\dots,
x_{\sigma(i)}, x_{q_1},\dots, x_{q_t}].
\end{equation}
We can assume that
$$
\{q_1,\dots, q_t\}\cap \{j_1,\dots, j_i\}=\varnothing
$$
Otherwise, the bracket is trivial. Observe that, the permutation
$$
\{\sigma(1),\dots, \sigma(i), q_1,\dots, q_t\}
$$
is from $\Sigma_{l+1}^{i+t}$ on the set $\{j_1,\dots, j_i;
q_1,\dots, q_t\},$ i.e. it is divided into $l+1$ monotonic blocks.
The number $d_l(\sigma)$ is the number of divisions of
$\{\sigma(1),\dots, \sigma(i), q_1,\dots, q_t\}$, which fixes the
last monotonic block $(q_1,\dots, q_t)$. Observe that, the number
of appearances of the bracket $[x_{n+1},x_{\sigma(1)},\dots,
x_{\sigma(i)}, x_{q_1},\dots, x_{q_t}]$ in the full product
(\ref{bigprod}) is exactly $d_{l+1}(\{\sigma(1),\dots, \sigma(i),
q_1,\dots, q_t\})$. The needed expression for the case $l+1$
follows.
\end{proof}

Note that, one can present the product from (\ref{decomposition})
in terms of shuffles as follows
\begin{multline*}\prod_{\sigma \in \Sigma_l^i, \sigma \subseteq \{1,\dots,
n\}}[x_{n+1}, x_{\sigma(1)},\dots, x_{\sigma(i)}]^{d_l(\sigma)}=\\
\prod_{i_1+\dots+i_l=i,\ \sigma\in [i_1,\dots, i_l]-{\rm
shuffles}} [x_{n+1}, x_{\sigma(1)},\dots,
x_{\sigma(i)}].\end{multline*}

Denote $K:=K_{n+1}^{\mathbb Z/p^r}$.
\begin{lemma}\label{divisible}
For $l\geq 2$,
$$
[x_{n+1},_lx_1\dots x_n]\in
\gamma_2(K)^{(l-1)!}\gamma_2\gamma_2(K).
$$
\end{lemma}
\begin{proof}
Denote $\tau_i(q)=\sum_{\sigma\in \Sigma_l^i(q)}d_l(\sigma).$
Since, modulo $\gamma_2\gamma_2(K),$ we can permute all letters in
the brackets in (\ref{decomposition}) except first two, we have
\begin{multline}
[x_{n+1},_lx_1\dots x_n]\equiv\\ \prod_{i=l}^n \prod_{j_1<\dots
<j_s<q<j_{s+1}<\dots <j_{l}} [x_{n+1},x_q, x_{j_1},\dots,
x_{j_l}]^{\tau_i(s+1)}\mod \gamma_2\gamma_2(K)
\end{multline}
By lemma \ref{lemma2}, all numbers $\tau_i(s+1)$  are divided by
$(l-1)!$ and the result follows.
\end{proof}

\begin{lemma}\label{lemma2.6}
For any $n\geq 1$, and $r>1$, $(x_1\dots x_n)^{p^r}\in
\gamma_2(K_n^{\mathbb
Z/p^r})^{p^{r-1}}\gamma_2\gamma_2(K_n^{\mathbb Z/p^r}).$
\end{lemma}
\begin{proof}
We prove by induction on $n$. For $n=1$, $x_1^{p^r}=1$. Assume
that the needed property holds for a given $n$ and prove it for
$n+1$. By lemma \ref{firstlemma},

\begin{multline}\label{pr}
(x_1\dots x_{n+1})^{p^r}=(x_1\dots
x_n)^{p^r}x_{n+1}^{p^r}\prod_{i}[x_{n+1},_{i-1}(x_1\dots x_n)]^{\binom{p^r}{i}}=\\
(x_1\dots x_n)^{p^r}\prod_{p|i}[x_{n+1},_{i-1}(x_1\dots
x_n)]^{\binom{p^r}{i}}.
\end{multline}

Using the equality (\ref{pr}), for the inductive step, it is
enough to prove that
$$
\prod_{p|i}[x_{n+1},_{i-1}(x_1\dots
x_n)]^{\binom{p^r}{i}}\in\gamma_2(K)^{p^{r-1}}\gamma_2\gamma_2(K)
$$
Given $i$, present it as $i=p^ze,\ (e,p)=1$. Moreover, we can
assume that $z\geq 1,$ since otherwise the whole bracket vanishes.
It remains to show that
\begin{equation}\label{remain}
[x_{n+1},_{i-1}(x_1\dots
x_n)]\in\gamma_2(K)^{p^{z-1}}\gamma_2\gamma_2(K).
\end{equation}
This follows from lemma \ref{divisible}, since $(i-1)!$ is
divisible by $p^{z-1}$. This proves (\ref{remain}) and finishes
the inductive step.
\end{proof}

For a subgroup $H$ of $K$, we denote by $[x_{n+1}, H]$ the
subgroup of $K$, generated by elements $[x_{n+1}, h],\ h\in H$.

\begin{prop} \label{proposition2.7}
For $n\geq 1$ and $r>1$,
$$
(x_1\dots x_{n+1})^{p^{r+1}}=(x_1\dots x_n)^{p^{r+1}}\gamma,
$$
where \begin{equation}\label{np2} \gamma\in
\gamma_2\gamma_2\gamma_2(K)[\gamma_2(K)^p,\gamma_2\gamma_2(K)](\gamma_2\gamma_2(K))^p
\end{equation}
as well as
\begin{equation}\label{np1}
\gamma\in\EuScript B_{n+1}[\EuScript B_{n+1},
\gamma_2(K)^p][\EuScript B_{n+1}, \gamma_2\gamma_2(K)].
\end{equation}
\end{prop}
\begin{proof}
$[$One of the key points of the proof of this proposition is the
possibility to permute the elements from $[x_{n+1},K]$. This
possibility covers the problems which appear due to non-normality
of the subgroup $\EuScript B_{n+1}$.$]$

It follows from (\ref{pr}) and the proof of the previous lemma
that
$$ (x_1\dots x_{n+1})^{p^r}=(x_1\dots x_n)^{p^r}\alpha,
$$
where $\alpha\in[x_{n+1},K_n^{\mathbb Z/p^r}]^{p^{r-1}}
(\gamma_2\gamma_2(K_n^{\mathbb Z/p^r})\cap [x_{n+1},K]).$ Taking
the $p$th power of $(x_1\dots x_n)^{p^r}\alpha$, we get
$$
(x_1\dots x_{n+1})^{p^{r+1}}=(x_1\dots
x_n)^{p^{r+1}}\alpha^p\beta,
$$
where \begin{equation}\label{beta}\beta\in [[x_{n+1},K]^{p^{r-1}}
(\gamma_2\gamma_2(K_n^{\mathbb Z/p^r})\cap [x_{n+1},K]),
\gamma_2(K_n^{\mathbb
Z/p^r})^{p^{r-1}}\gamma_2\gamma_2(K_n^{\mathbb
Z/p^r})].\end{equation}
The needed element $\gamma$ is
$\alpha^p\beta$. Present $\alpha$ as $\alpha=\alpha_1\alpha_2,$
where
\begin{align*} & \alpha_1\in [x_{n+1},K]^{p^{r-1}},\\ & \alpha_2\in(\gamma_2\gamma_2(K_n^{\mathbb Z/p^r})\cap
[x_{n+1},K]).
\end{align*}
The elements $\alpha_1$ and $\alpha_2$ commute, since they lie in
$[x_{n+1}, K]$. Observe that $\alpha_1^p=1,$ since
$$
[x_{n+1},K]^{p^r}=1.
$$
For an element $\alpha_2$, we have $\alpha_2\in
\gamma_2\gamma_2(K)$, therefore,
$$
\alpha^p\in \gamma_2\gamma_2(K)^p\gamma_2\gamma_2\gamma_2(K).
$$
Together with (\ref{beta}), we have a needed result (\ref{np2}).

Now we will prove (\ref{np1}). First consider the element
$\alpha_2$. It was already observed that $\alpha^p=\alpha_2^p$.
The element $\alpha_2$ is a product of elements of the form (and
their inverses)
$$
[[x_{i_1},\dots, x_{i_t}],[x_{j_1},\dots, x_{j_s}]],
$$ where one of the generators in this brackets is $x_{n+1}$. If
$t+s$ is not a power of $p$, then this bracket lies in $\EuScript
B_{n+1}$, and we can move it to the term $\EuScript B_{n+1}$ in
(\ref{np1}). If $t+s$ is a power of $p$, then one of $t$ or $s$
must not be a power of $p$, assume it is $t$. Then,
\begin{multline*}[[x_{i_1},\dots, x_{i_t}],[x_{j_1},\dots,
x_{j_s}]]^p=\\ [[x_{i_1},\dots, x_{i_t}],[x_{j_1},\dots,
x_{j_s}]^p]\in [\EuScript B_{n+1},
\gamma_2(K)^p]\cap[x_{n+1},K].\end{multline*} Now we consider the
element $\beta$, which is a product of certain brackets from the
subgroup $[x_{n+1}, K].$ These brackets (or
their inverses) have one of the following forms:\\
(a) $[\beta_1, \beta_2],$ where $\beta_1$ and $\beta_2$ are of
commutators in generators $x_i$-s, $\beta_1,\beta_2\in
\gamma_2\gamma_2(K);$\\
(b) $[\beta_1,\beta_2],$ where $\beta_1=\delta^{p^{r-1}}$, where
$\delta$ is some commutator in generators and $\beta_2$ is some
commutator
in generators from $\gamma_2\gamma_2(K);$\\
(c) $[\beta_1,\beta_2],$ where $\beta_i=\delta_i^{p^{r-1}},\
i=1,2$ and $\delta_i$ are some commutators in generators.

For a commutator in generators $\xi$, denote by $|\xi|$ its
commutator length, i.e. the number of the maximal term of the
lower central series where $\xi$ lies. Consider the case (a). If
$|\beta_1|+|\beta_2|$ is not a power of $p$, then the bracket
$[\beta_1,\beta_2]$ lies in $\EuScript B_{n+1}\cap [x_{n+1},K]$.
Suppose that $|\beta_1|+|\beta_2|$ is a power of $p$. Then, one at
least one of $|\beta_1|$ or $|\beta_2|$ is not a power of $p$, say
$\beta_1$. Then $\beta_1\in \EuScript B_{n+1}$ and, therefore,
$[\beta_1,\beta_2]\in [\EuScript B_{n+1}, \gamma_2\gamma_2(K)].$
The same situation is in the case (b). If we assume that
$|\beta_1|$ is not a power of $p$, we obtain an element from
$[\EuScript B_{n+1}, \gamma_2\gamma_2(K)]$, if we assume that
$|\beta_2|$ is not a power of $p$, we obtain an element from
$[\EuScript B_{n+1}, \gamma_2(K)^{p^{r-1}}].$ In the same way we
can handle the case (c). Observe also that, since $r>1$, the case
(c) becomes trivial, since
$$[\beta_1,\beta_2]=[\delta_1^{p^{r-1}},\delta_2^{p^{r-1}}]=[\delta_1^{p^r},\delta_2^{p^{r-2}}]=1.$$
Since all brackets which we consider lie in $[x_{n+1}, K]$, we can
permute them. This argument shows that the element $\gamma$
satisfies the needed property (\ref{np1}).
\end{proof}

\begin{cor}\label{corollary2.8}
For any $n\geq 1$ and $r>1$,
$$
(x_1\dots x_n)^{p^{r+1}}\in \gamma_2\gamma_2\gamma_2(K_n^{\mathbb
Z/p^r})[\gamma_2(K_n^{\mathbb
Z/p^r})^p,\gamma_2\gamma_2(K_n^{\mathbb
Z/p^r})](\gamma_2\gamma_2(K_n^{\mathbb Z/p^r}))^p
$$
\end{cor}
Observe that, for $r=1,$ the situation is different. In this case,
$$
(x_1\dots x_n)^{p^3}\in \gamma_2\gamma_2\gamma_2(K_n^{\mathbb
Z/p}),
$$
what can be easily proved by induction on $n$.

\section{The geometric candidates for the subgroup $\EuScript{B}_n$ of $K_n$}\label{section3}

The candidates from the subgroup $\ScriptB_n$ of $K_n$ can be obtained from functorial decompositions of the loop-suspension functor on path-connected $p$-local co-$H$-spaces.
Let us recall some results from~\cite{Selick-Wu,Selick-Wu2}. Let $V$ be a module over the field $\Z/p$. The tensor algebra $T(V)$ is a Hopf algebra by saying $V$ primitive. With forgetting the algebra structure, we have the functor $T$ from modules to coalgebras. According to~\cite{Selick-Wu}, there are functors $B^{\max}$ and $A^{\min}$ from modules to coalgebras with the properties
\begin{enumerate}
\item[1)] $A^{\min}$ is an indecomposable functor from modules to coalgebras;
\item[2)]  there is a functorial coalgebra isomorphism
\begin{equation}\label{equation3.1}
T(V)\cong B^{\max}(V)\otimes A^{\min}(V)
\end{equation}
with $V\subseteq A^{\min}(V)$.
\end{enumerate}
Here $B^{\max}(V)$ can be chosen a functorial sub Hopf algebra of $T(V)$ with a left functorial coalgebra inverse. According to~\cite[Section 2]{Selick-Wu2}, the functorial coalgebra decomposition~(\ref{equation3.1}) holds over $p$-local integers. From this, \cite[Theorem 1.5]{Selick-Wu} can be extended over $p$-local integers and so we have an important property on the Lie powers of tensor length ~$n$
\begin{equation}\label{equation3.2}
L_n(V)\subseteq B^{\max}(V) \textrm{ if } n \textrm{ is not a power of } p
\end{equation}
for any free module $V$ over $p$-local integers. (\textbf{Note.} Property~(\ref{equation3.2}) holds for any choice of the functor $B^{\max}$.)

The algebraic functors $A^{\min}$ and $B^{\max}$ admits geometric realization in the sense of~\cite{Selick-Wu,Selick-Wu2} that there are homotopy functors $A^{\min}$ and $Q^{\max}$ from path-connected $p$-local co-$H$-spaces to spaces with the following properties
\begin{enumerate}
\item[1)] $Q^{\max}_n(X)$ is a functorial retract of $\Sigma X^{\wedge n}$.
\item[2)] There is a functorial fibre sequence
$$
A^{\min}(X)\rTo^{j_X} \bigvee_{n=2}^{\infty}Q^{\max}_n(X)\rTo^{\pi_X}\Sigma X
$$
with $j_X\simeq \ast$. Here, the map $\pi_X$ is given as a composite
\begin{equation}\label{equation3.3}
\pi_X\colon Q^{\max}_n(X)\rInto \Sigma X^{\wedge n}\rTo^{W_n} \Sigma X,
\end{equation}
where $W_n$ is the Whitehead product.
\item[3)] There is a functorial decomposition
$$
\Omega\Sigma X\simeq A^{\min}(X)\times\Omega (\bigvee_{n=2}^{\infty}
Q^{\max}_n(X)).
$$
\item[4)] Let $B^{\max}(X)=\Omega (\bigvee_{n=2}^{\infty}
Q^{\max}_n(X))$. Then the mod $p$ homology
$$
H_*(A^{\min}(X))\cong A^{\min}(\tilde H_*(X))\textrm{ and } H_*(B^{\max}(X))\cong B^{\max}(\tilde H_*(X)).
$$
\end{enumerate}
(\textbf{Note.} The geometric functors $A^{\min}$ and $B^{\max}$ can be generalized for decomposing any looped co-$H$-spaces~\cite{STW1,STW2}. Here we are only interested in the case that $A^{\min}(X)$ and $B^{\max}(X)$ for co-$H$-spaces $X$.)

\begin{prop}\label{B^{max}-property}
Let $X$ be any path-connected $p$-local co-$H$-space. Then there is a homotopy commutative diagram
\begin{diagram}
B^{\max}(X)&\rTo^{\Omega \pi_X}&\Omega \Sigma X\\
\uDashto&\ruTo>{S_n}&\\
 X^{\wedge n}&&\\
\end{diagram}
for $n$ not a power of $p$, where $S_n$ is the $n$-fold Samelson product.
\end{prop}
\begin{proof}
The assertion follows from Property~(\ref{equation3.2}) and the following commutative diagram
\begin{diagram}
[X^{\times n}, B^{\max}(X)] &\lTo& \mathcal{B}^{\max}_n&\rTo^{\cong}& \coalg(C(\ - \ )^{\otimes n}, B^{\max}(\ - \ ) ) \\
\dInto&&\dInto&\textrm{ pull-back}&\dInto\\
[X^{\times n}, \Omega\Sigma X]&\lTo^{e_X}& K_n^{\Z_{(p )}}& \rTo^{\cong}& \coalg( C(\ - \ )^{\otimes n}, T( \ - \ ),\\
\end{diagram}
where $C(V)=V\oplus\Z_{(p )}$ with trivial comultiplication as a functor from free $Z_{(p )}$-modules to coalgebras,  the terms in the right column means the group of natural coalgebra transformations,  $K_n^{\Z_{(p )}}=K_n^{\Z_{(p )}}(x_1,\ldots,x_n)$ is the Cohen group over $p$-local integers~\cite[Section 1.4]{Wu} and $e_X$ is the representation of the Cohen group on $[X^{\times n}, \Omega\Sigma X)]$ which sends $x_i$ to the homotopy class of the composite
$$
X^{\times n} \rTo^{i-\textrm{th coordinate projection}} X\rInto \Omega \Sigma X.
$$
\end{proof}

The groups $\mathcal{B}^{\max}_n$ defined as above are the candidates for the subgroup $\ScriptB_n$ of the Cohen group $K_n$ over $\Z_{(p )}$ or $\Z/p^r$ with the desired property that any commutator of
length $\neq p^t, t\geq 0$ whose entrances are generators, is in
$\EuScript B_n$. For a given co-$H$-space $X$, the $B^{\max}(X)$ can be a starting candidate for  producing the subgroups $\ScriptB_n$. The derived series discussed in section~\ref{section2} occurs naturally for resolutions of co-$H$-spaces by fibrations into $H$-spaces with the following observation. Let $Y$ be an $H$-space. Let $f\colon \Sigma X\to Y$ be a map with a fibre sequence $F_f\rTo^j \Sigma X\rTo^f Y$, where $F_f$ is the homotopy fibre of $f$. Then $\gamma_2([Z,\Omega\Sigma X])\leq \im(\Omega j_*\colon [Z,\Omega F_f]\to [Z,\Omega\Sigma X]$ for any spaces $Z$. By taking another map $f_1$ from $F_f$ to an $H$-space $Y_1$ with the homotopy fibre $F_{f_1}$, then  $\gamma_2\gamma_2([Z,\Omega\Sigma X])$ lies in the image from $[Z,\Omega F_1]$. Since the $H$-space resolutions for co-$H$-spaces seem out of control under current technology, we concentrate on the discussions on Moore spaces for highlighting the ideas of combinatorial approach in homotopy theory in next sections.

\section{Applications to the Moore Spaces}\label{section4}
 Let us consider the Moore space $P^{2n+1}(p^r)$ with $n>1$ and $p>3$.  The hypothesis $n>1$ is used so that $P^{2n}(p^r)$ is a co-$H$-space, and the hypothesis $p>3$ is used so that the mod $p^r$ homotopy groups $\pi_*(\Omega P^{2n+1}(p^r);\Z/p^r)$ is a Lie algebra~\cite[Proposition 6.2]{CMN}. Recall from~\cite{CMN3}, there is a fibre sequence
 \begin{equation}\label{equation4.1}
 T^{2n+1}\{p^r\}\rTo^{j}P(n,p^r)\rTo^{\tilde\pi} P^{2n+1}(p^r).
 \end{equation}
where $T^{2n+1}\{p^r\}$ is the atomic piece of $\Omega P^{2n+1}(p^r)$ containing the bottom cell for $n>1$, the map $j$ is null homotopic, $P(n,p^r)$ is a wedge of mod $p^r$ Moore spaces given as a retract of $\bigvee\limits_{k=2}^\infty \Sigma (P^{2n}(p^r))^{\wedge k}$ and the map $\tilde\pi$ is given as a composite
\begin{equation}\label{equation4.2}
\tilde\pi\colon  P(n,p^r)\rInto \bigvee\limits_{k=2}^\infty \Sigma (P^{2n}(p^r))^{\wedge k}\rTo^{\bigvee\limits_{k=2}^\infty W_k} P^{2n+1}(p^r)
\end{equation}
with $W_k$ the iterated Whitehead product. Let $\partial\colon \Omega P^{2n+1}(p^r)\to T^{2n+1}\{p^r\}$ be the connecting map of the fibre sequence~(\ref{equation4.1}). Since $j$ is null homotopic, the map
$$
\psi=(\partial, \tilde\pi)\colon \Omega P^{2n+1}(p^r)\simeq T^{2n+1}\{p^r\}\times \Omega P(n,p^r)
$$
is a homotopy equivalence.

\begin{theorem}\label{theorem4.1}
The composite
$$
\Omega P^{2n+1}(p^r)\rTo^{p^{r+1}}\Omega P^{2n+1}(p^r)\rTo^{\partial} T^{2n+1}\{p^r\}
$$
is null homotopic for $p>3$, $n>1$ and $r>1$.
\end{theorem}

Some preliminary settings are required before we prove this theorem. Recall that the mod $p$ homology $H_*(\Omega
P^{2n+1}(p^r))=T(V)$ as a Hopf algebra, where $V=\tilde
H_*(P^{2k}(p^r))$ having a basis $\{u,v\}$ with $|v|=2n$,
$|u|=2n-1$ and  the $r$-th Bockstein $\beta^rv=u$. Under the
hypothesis that $n>1$, $H_*(\Omega P^{2n+1}(p^r))=T(u,v)$  is a
primitively generated Hopf algebra. In any Lie algebra $L$ with $x,y\in L$, $\ad^0(y)(x)=x$ and $\ad^k(y)(x)=[x,\ad^{k-1}(y)(x)]$ for $k\geq1$. Let
$$
\tau_k=\ad^{p^k-1}(v)(u)\textrm{ and }
\sigma_k=\sum_{j=1}^{p^k-1}\frac{1}{2p}\binom{p^k}{j}
[\ad^j(v)(u), \ad^{p^k-j}(v)(u)].
$$
By~\cite{CMN3}, the mod $p$ homology $H_*(T^{2n+1}\{p^r\})$ is isomorphic to the free graded commutative algebra generated by $u,v,\tau_k,\sigma_k$ with $k\geq1$ as a graded coalgebra. Let $L(V)\subseteq T(V)$ be the free graded Lie algebra generated by $V$. From the fibre sequence
$$
\Omega P(n,p^r)\rTo^{\Omega\tilde\pi}\Omega P^{2n+1}(p^r)\rTo^{\partial} T^{2n+1}\{p^r\},
$$
the sub Lie algebra
$$
L(P(n,p^r))=L(V)\cap \im(\Omega\tilde\pi_*\colon H_*(\Omega P(n,p^r))\to H_*(\Omega P^{2n+1}(p^r)))
$$
can be described by the following diagram
\begin{equation}\label{equation4.3}
\begin{diagram}
L(P(n,p^r))&   &&\\
\dInto&\rdInto&&\\
[L(V),L(V)]&\rInto &L(V)&\rOnto&L(V)^{\ab}\\
\dOnto&&&\\
\sum\limits_{k=1}^\infty L(\tau_k,\sigma_k)^{\ab},&&&\\
\end{diagram}
\end{equation}
where the row and the column are short exact sequences of graded Lie algebras and $\sum\limits_{k=1}^\infty L(\tau_k,\sigma_k)^{\ab}$ is the product of the abelian graded Lie algebras.
The mod $p$ homology
$$
\tilde H_*(P(n,p^r))\cong \Sigma L(P )^{\ab},
$$
the suspension of the module $L(P )^{\ab}$, and
$$
H_*(\Omega P(n,p^r))\cong U(L(P ))\cong T(L(P )^{\ab},
$$
where $U(L)$ is the universal enveloping algebra of a Lie algebra $L$.

Let $K_k(P )$ be the subgroup of $[(P^{2n}(p^r))^{\times k}, \Omega P^{2n+1}(p^r)]$ generated by the homotopy classes of the composites
$$
x_i(P )\colon (P^{2n}(p^r))^{\times k}\rTo^{\pi_i} P^{2n}(p^r)\rInto \Omega P^{2n+1}(p^r),
$$
where $\pi_i$ is the $i$-th coordinate projection. Let
\begin{equation}\label{equation4.4}
\ScriptB_k(P )=K_q(P )\cap\im(\Omega\tilde\pi_*\colon[(P^{2n}(p^r))^{\times k}, \Omega P(n,p^r)]\to [(P^{2n}(p^r))^{\times k}, \Omega P^{2n+1}(p^r)]).
\end{equation}

\begin{lemma}\label{lemma4.2}
With the notations as above,
$\gamma_2(\gamma_2(K_k( P))\leq \ScriptB_k(P )$ for each $k\geq2$.
\end{lemma}
\begin{proof}
Recall that if $f\colon P^s(p^r)\to \Omega X$ and $g\colon P^t(p^r)\to \Omega X$.
According to~\cite[(5.8) and (5.9)]{Neisendorfer3}, the usual Samelson product $[f,g]\colon P^s(p^r)\wedge P^t(p^r)\to \Omega X$ decomposes as two maps
\begin{equation}\label{equation4.5}
[f,g]\colon P^{s+t}(p^r)\to P^{s+t}(p^r)\vee P^{s+t-1}(p^r)\simeq  P^s(p^r)\wedge P^t(p^r)\rTo^{[f,g]} \Omega X,
\end{equation}
which is called the mod $p^r$ Samelson product, and
\begin{equation}\label{equation4.6}
\{f,g\}\colon P^{s+t-1}(p^r)\to P^{s+t-1}(p^r)\vee P^{s+t-1}(p^r)\simeq  P^s(p^r)\wedge P^t(p^r)\rTo^{[f,g]} \Omega X
\end{equation}
with $\{f,g\}=[\beta^rf,g]+(-1)^{a+1}[f,\beta^rg]$, where $\beta^r$ is the Bockstein operation in the sense of ~\cite{Neisendorfer0}.
Observe that the mod $p^r$ homology $H_*(\Omega P^{2n+1};\Z/p^r)$ is a free $\Z/p$-module with $H_*(\Omega P^{2n+1};\Z/p^r)=T(u_r,v_r)$ as a Hopf algebra with $|u_r|=2n-1$ and $|v_r|=2n$. Following~\cite{CMN}, let $\mu\in \pi_{2n-1}(\Omega P^{2n+1};\Z/p^r)$ and $\nu\in \pi_{2n}(\Omega P^{2n+1};\Z/p^r)$ be the elements in mod $p^r$ homotopy groups whose Hurewicz image are given by $u_r$ and $v_r$, respectively. Since the Hurewicz homomorphism
$$
H\colon \pi_*(\Omega P^{2n+1}(p^r);\Z/p^r)\longrightarrow H_*(\Omega P^{2n+1}(p^r);\Z/p^r)
$$
is a morphism of graded Lie algebras, the sub Lie algebra of $\pi_*(\Omega P^{2n+1}(p^r);\Z/p^r)$ generated by $\mu,\nu$ is a free Lie algebra $L(\mu,\nu)$, which embeds into mod $p^r$ homology under the Hurewicz homomorphism. By formulae~(\ref{equation4.5}) and ~(\ref{equation4.6}), the iterated Samelson product
$$
S_t\colon (P^{2n}(p^r))^{\wedge t}\longrightarrow \Omega P^{2n+1}(p^r)
$$
decomposes as a linear combination of Lie elements in $L(\mu,\nu)$ for $t\geq1$. Let
$$
\tilde\pi'\colon \Sigma^{-1}P(n,p^r)\longrightarrow \Omega P^{2n+1}(p^r)
$$
be the adjoint map of $\tilde\pi$. By  definition~(\ref{equation4.2}),  the homotopy class of the map $\tilde\pi'$ restricted to each factor of mod $p^r$ Moore spaces in $\Sigma^{-1}P(n,p^r)$ is given by an element in $L(\mu,\nu)$. Let $\tilde L(P(n,p^r))$ be the sub Lie algebra of $\pi_*(\Omega P^{2n+1}(p^r))$ generated by the homotopy classes of the map $\tilde\pi'$ restricted to each factor of mod $p^r$ Moore spaces in $\Sigma^{-1}P(n,p^r)$. Then
\begin{equation}\label{equation4.7}
\tilde L(P(n,p^r))\subseteq \im(\Omega\tilde\pi_*\colon \pi_*(\Omega P(n,p^r);\Z/p^r)\to \pi_*(\Omega P^{2n+1}(p^r);\Z/p^r)).
\end{equation}
By using the property that the Hurewicz homomorphism to the mod $p^r$ homology restricted to $L(\mu,\nu)$ is injective, sub Lie algebra $\tilde L(P(n,p^r))$ of $L(\mu,\nu)$ can be described by diagram~(\ref{equation4.3}) with that $L(V)$ is replaced by $L(\mu,\nu)$, $L(P(n,p^r))$ is replaced by $\tilde L(P(n,p^r))$, and $\tau_k,\sigma_k$ are replaced by their corresponding Lie elements in $L(\mu,\nu)$.  It follows that
\begin{equation}\label{equation4.8}
[[L(\mu,\nu),L(\mu,\nu)], [L(\mu,\nu), L(\mu,\nu)]]\leq \tilde L(P(n,p^r))
\end{equation}
Observe that the subgroup $\gamma_2(\gamma_2(K_k(P )))$ is generated by the commutators
$$
x_{I,J}=[[[x_{i_1}( P),x_{i_2}(P )], \ldots, x_{i_s}( P)],[[x_{j_1}( P),x_{j_2}(P )],\ldots,x_{j_t}(P )]
$$
for $1\leq i_1,\ldots,i_s,j_1,\ldots,j_t\leq k$. Note that the geometric interpretation of the commutator $[[x_{i_1}( P),x_{i_2}(P )], \ldots, x_{i_s}( P)]$ is the homotopy class of the composite
$$
(P^{2n}(p^r))^{\times q}\rTo^{\pi_I} (P^{2n}(p^r))^{\wedge s} \rTo^{S_s}\Omega P^{2n+1}(p^r),
$$
where $\pi_I$ is given as a composite of a coordinate projection $(P^{2n}(p^r))^{\times k}\to (P^{2n}(p^r))^{\times s}$ followed by the pinch map $(P^{2n}(p^r))^{\times s}\to (P^{2n}(p^r))^{\wedge s}$. By using the property that $S_s$ decomposes as a linear combination of Lie elements in $L(\mu,\nu)$ together with properties ~(\ref{equation4.7}) and ~(\ref{equation4.8}), we have
$$
x_{I,J}\in\im (\Omega\tilde\pi_*\colon [(P^{2n}(p^r)^{\times k}, \Omega P(n,p^r)]\to [(P^{2n}(p^r)^{\times k}, \Omega P^{2n+1}(p^r)]).
$$
The assertion follows.
\end{proof}

 \begin{proof}[Proof of Theorem~\ref{theorem4.1}]
 Let $J(X)$ be the James construction on a pointed space $X$ with the James filtration $J_k(X)$. Let $q_k\colon X^{\times k}\to J_k(X)$ be the projection map and let
 $$
 d^i\colon X^{\times k-1}\to X^{\times k}, \quad (x_1,\ldots,x_{k-1})\mapsto (x_1,\ldots,x_{i-1},\ast,x_i,\ldots,x_{k-1})
 $$
 be the coordinate inclusion for $1\leq i\leq k$. Let $\calH_k(X,\Omega Y)$ be the equalizer of the group homomorphisms
 $$
 d^{i*}\colon [X^{\times k}, \Omega Y]\longrightarrow [X^{\times k-1},\Omega Y]
 $$
 for $1\leq i\leq k$. By~\cite[Theorem 1.1.5]{Wu}, $q_k^*\colon [J_k(X),\Omega Y]\to [X^{\times k},\Omega Y]$ is a group monomorphism with its image given by $\calH_k(X,\Omega Y)$. Moreover the inclusion $J_{k-1}(X)\to J_k(X)$ induces a group epimorphism $[J_k(X),\Omega Y]\rOnto^{}_{}{ [J_{k-1}(X),\Omega Y]}$ with
 $$
 [J(X),\Omega Y]\cong\lim_k [J_k(X),\Omega Y]\cong \lim_k\calH_k(X,\Omega Y)
 $$
 being given by the inverse limit. We identify the group $[J_k(X),\Omega Y]$ with its image in $[X^{\times k},\Omega Y]$ under group monomorphism $q_k^*$ and the group $[J(X),\Omega Y]$ with the inverse limit $\calH(X,\Omega Y)=\lim\limits_k\calH_k(X,\Omega Y).$

 For any pointed space $X$, we identify the group $[X, \Omega P(n,p^r)]$ with its image in $[X, \Omega P^{2n+1}(p^r)]$ under the group monomorphism
 $$
 \Omega\tilde\pi_*\colon [X, \Omega P(n,p^r)]\rInto^{}_{} {[X, \Omega P^{2n+1}(p^r)]}.
 $$
Let $\alpha_k=x_1( P)\cdots x_k( P)\in K_k( P)$. By Corollary~\ref{corollary2.8} and Lemma~\ref{lemma4.2}, we have
$$
\alpha_k^{p^{r+1}}\in \ScriptB_k( P)
$$
for each $k$. Since $\alpha_k^{p^{r+1}}\in\calH_k(P^{2n}(p^r), \Omega P(n,p^r))$, we have
$$
\alpha_k^{p^{r+1}}\in \ScriptB_k( P)\cap \calH_k(P^{2n}(p^r), \Omega P(n,p^r)).
$$
With letting $k\to\infty$, we obtain a map
$$
f\colon J(P^{2n}(p^r))\longrightarrow \Omega P(n,p^r)
$$
such that the composite $(\Omega\tilde\pi)\circ f$ represents the homotopy class $$\alpha_{\infty}^{p^{r+1}}\in \calH(P^{2n}(p^r),\Omega P^{2n+1}(p^r))\cong[J(P^{2n}(p^r), \Omega P^{2n+1}(p^r)]$$
whose geometric interpretation is the power map $$p^{r+1}\colon J(P^{2n}(p^r))\simeq\Omega P^{2n+1}(p^r)\to \Omega P^{2n+1}(p^r).$$
Thus there is a homotopy commutative diagram
\begin{diagram}
& &\Omega P(n,p^r)\\
&\ruTo>{f}&\dTo>{\Omega\tilde\pi}\\
\Omega P^{2n+1}(p^r)&\rTo^{p^{r+1}}&\Omega P^{2n+1}(p^r),\\
\end{diagram}
and hence the result follows.
 \end{proof}

\section{Applications to the Anick spaces}\label{section5}
Let $E^{2n+1}\{p^r\}$ be the homotopy
fibre of the inclusion map $P^{2n+1}\{p^r\}\to S^{2n+1}\{p^r\}$,
where $S^{2n+1}\{p^r\}$ is the homotopy fibre of the
degree map $[p^r]\colon S^{2n+1}\to S^{2n+1}$.  Let $F^{2n+1}\{p^r\}$ be the homotopy fibre of the pinch map $P^{2n+1}(p^r)\to S^{2n+1}$. Then there is a homotopy commutative diagram of fibre sequences
\begin{equation}\label{E-fibration}
\begin{diagram}
F^{2n+1}\{p^r\}&\rTo&P^{2n+1}(p^r)&\rTo& S^{2n+1}\\
\uTo&&\uEq&&\uTo\\
E^{2n+1}\{p^r\}&\rTo^{\sigma}&P^{2n+1}(p^r)&\rTo^{\phi}&S^{2n+1}\{p^r\}\\
\uTo>{j}&&\uTo&&\uTo\\
\Omega^2S^{2n+1}&\rTo&\ast&\rTo&\Omega S^{2n+1}.\\
\end{diagram}
\end{equation}
Let $W_n$ be the homotopy theoretic fibre of the double suspension $S^{2n-1}\to \Omega^2S^{2n+1}$. The space $W_n$ is deloopable and its classifying space $BW_n$ is an $H$-space~\cite{Gray} with a fibre sequence
$$
S^{2n-1}\rTo \Omega^2S^{2n+1}\rTo^{\nu} BW_n.
$$
By~\cite[Corollary 3.5]{Gray-Theriault}, the Gray map $\nu$ factors through $E^{2n+1}\{p^r\}$ with a homotopy commutative diagram
\begin{equation}\label{map-nu^E}
\begin{diagram}
E^{2n+1}\{p^r\}&\rTo^{\nu^E}&BW_n\\
\uTo>{j}&\ruTo>{\nu}&\\
\Omega^2S^{2n+1}.&&\\
\end{diagram}
\end{equation}
Let $R_0$ be the homotopy fibre of $\nu^E\colon E^{2n+1}\{p^r\}\to BW_n$. By~\cite[Theorem 3.8]{Gray-Theriault}, there is a homotopy commutative diagram of fibre sequences
\begin{diagram}
T^{2n-1}_{\infty}(p^r)&\rTo&\Omega S^{2n+1}\{p^r\}& \rTo& BW_n\\
\dTo&&\dTo&&\dEq\\
R_0&\rTo^{\sigma_1}&E^{2n+1}\{p^r\}&\rTo^{\nu^E}& BW_n\\
\dTo>{\sigma\circ\sigma_1}&&\dTo>{\sigma}&&\\
P^{2n+1}(p^r)&\rEq&P^{2n+1}(p^r),&&\\
\end{diagram}
where the Anick space $T^{2n-1}_{\infty}(p^r)$ is denoted as $T_{2n-1}$ in~\cite{Gray-Theriault}.
The left column gives a fibre sequence
\begin{equation}\label{R0}
\Omega R_0\rTo^{\Omega(\sigma\circ\sigma_1)} \Omega P^{2n+1}(p^r)\rTo^{\partial} T^{2n-1}_{\infty}(p^r),
\end{equation}
where $\partial$ is the connecting map as in~\cite[Corollary 3.9]{Gray-Theriault}.

\begin{theorem}\label{theorem5.1}
The composite
$$
\Omega P^{2n+1}(p^r)\rTo^{p^{r}}\Omega P^{2n+1}(p^r)\rTo^{\partial} T^{2n-1}_{\infty}\{p^r\}
$$
is null homotopic for $p>3$, $n>1$ and $r>1$.
\end{theorem}
\begin{proof}
The assertion follows by using the same arguments in the proof of Theorem~\ref{theorem4.1}. Here, we choose the subgroup
$$
\ScriptB_k(R_0 )=K_q(P )\cap\im(\Omega\sigma\circ\sigma_{1*}\colon[(P^{2n}(p^r))^{\times k}, \Omega R_0]\to [(P^{2n}(p^r))^{\times k}, \Omega P^{2n+1}(p^r)])
$$
with the property that $\gamma_2(K_k( P))\leq \ScriptB_k(R_0)$ by using the same arguments in the proof of Lemma~\ref{lemma4.2}.
\end{proof}
Together with ~\cite[Theorem 1]{Neisendorfer3}, the map $\partial \colon \Omega P^{2n+1}(p^r)\to T^{2n-1}_{\infty}\{p^r\}$ has a right homotopy inverse after looping, we have the following.

\begin{cor}\label{corollary5.2}
The space $\Omega T^{2n-1}_{\infty}\{p^r\}$ has a multiplicative exponent $p^r$. In particular,  $p^r\cdot \pi_*(T^{2n-1}_{\infty}\{p^r\})=0$.\hfill $\Box$

\end{cor}

\noindent\textbf{Note.} Corollary~\ref{corollary5.2} is ~\cite[Theorem 2]{Neisendorfer3}, where  $\Omega T^{2n-1}_{\infty}\{p^r\}$  was denoted as $D(n,r)$ in~\cite{Neisendorfer3}. Theorem~\ref{theorem5.1} improves ~\cite[Theorem 2]{Neisendorfer3} in the sense that the $p^r$ power map of $\Omega P^{2n+1}(p^r)$ already goes trivially to the Anick space up to homotopy before looping.

\vspace{.5cm}\noindent {\it Acknowledgements.} The authors thank
F. Petrov for discussions related to the subject of the paper.

\end{document}